\documentclass{elsarticle}

\usepackage{latexsym,enumerate}
\usepackage{amsmath,amsthm,amsopn,amstext,amscd,amsfonts,amssymb}
\usepackage[ansinew]{inputenc}
\usepackage{verbatim}
\usepackage{pstricks,pst-node}
\usepackage{wasysym}
\usepackage[all]{xy}
\usepackage{epsfig} 
\usepackage{graphicx,psfrag} 
\usepackage{subfigure}

\newcommand{\NN}{\mathbb{N}}

\newtheorem{theorem}{Theorem}[section]
\newtheorem{lemma}[theorem]{Lemma}
\newtheorem{proposition}[theorem]{Proposition}
\newtheorem{corollary}[theorem]{Corollary}
\newtheorem{definition}[theorem]{Definition}
\newtheorem{example}[theorem]{Example}
\newtheorem{remark}[theorem]{Remark}

\DeclareMathOperator{\diam}{diam}

\renewcommand{\a}{\alpha}

\renewcommand{\d}{\delta}
\newcommand{\D}{\Delta}
\newcommand{\g}{\gamma}
\newcommand{\G}{\Gamma}

\newcommand{\s}{\sigma}

\journal{Aequationes Mathematicae}

\begin{document}

\begin{frontmatter}

\title{Hyperbolicity in the corona and join of graphs}

\author[a]{Walter Carballosa\corref{x}}
\address[a]{Consejo Nacional de Ciencia y Tecnolog\'ia (CONACYT) $\&$ Universidad Aut\'onoma de Zacatecas,
Paseo la Bufa, int. Calzada Solidaridad, 98060 Zacatecas, ZAC, M\'exico}
\ead{wcarballosato@conacyt.mx} \cortext[x]{Corresponding author.}

\author[b]{Jos\'e M. Rodr{\'\i}guez}
\address[b]{Department of Mathematics, Universidad Carlos III de Madrid, Av. de la Universidad 30, 28911 Legan\'es, Madrid, Spain.}
\ead{jomaro@math.uc3m.es}

\author[d]{Jos\'e M. Sigarreta}
\address[d]{Faculdad de Matem\'aticas, Universidad Aut\'onoma de Guerrero, Carlos E. Adame 5, Col. La Garita, Acapulco, Guerrero, Mexico}
\ead{jsmathguerrero@gmail.com}

\begin{abstract}
If X is a geodesic metric space and $x_1,x_2,x_3\in X$, a {\it
geodesic triangle} $T=\{x_1,x_2,x_3\}$ is the union of the three
geodesics $[x_1x_2]$, $[x_2x_3]$ and $[x_3x_1]$ in $X$. The space
$X$ is $\delta$-\emph{hyperbolic} $($in the Gromov sense$)$ if any side
of $T$ is contained in a $\delta$-neighborhood of the union of the two
other sides, for every geodesic triangle $T$ in $X$.
If $X$ is hyperbolic, we denote by
$\delta(X)$ the sharp hyperbolicity constant of $X$, i.e.
$\delta(X)=\inf\{\delta\ge 0: \, X \, \text{ is $\delta$-hyperbolic}\,\}\,.$
Some previous works characterize the hyperbolic product graphs (for the Cartesian product, strong product and lexicographic product) in terms of properties of the factor graphs.
In this paper we characterize the hyperbolic product graphs for graph join $G_1\uplus G_2$ and the corona $G_1\diamond G_2$:
$G_1\uplus G_2$ is always hyperbolic, and $G_1\diamond G_2$ is hyperbolic if and only if $G_1$ is hyperbolic.
Furthermore, we obtain simple formulae for the hyperbolicity constant of the graph join $G_1\uplus G_2$ and the corona $G_1\diamond G_2$.
\end{abstract}


\end{frontmatter}

\section{Introduction}

Hyperbolic spaces play an important role in geometric
group theory and in the geometry of negatively curved
spaces (see \cite{ABCD,GH,G1}).
The concept of Gromov hyperbolicity grasps the essence of negatively curved
spaces like the classical hyperbolic space, Riemannian manifolds of
negative sectional curvature bounded away from $0$, and of discrete spaces like trees
and the Cayley graphs of many finitely generated groups. It is remarkable
that a simple concept leads to such a rich
general theory (see \cite{ABCD,GH,G1}).

The first works on Gromov hyperbolic spaces deal with
finitely generated groups (see \cite{G1}). 
Initially, Gromov spaces were applied to the study of automatic groups in the science of computation
(see, \emph{e.g.}, \cite{O}); indeed, hyperbolic groups are strongly geodesically automatic, \emph{i.e.}, there is an automatic structure on the group \cite{Cha}.

The concept of hyperbolicity appears also in discrete mathematics, algorithms
and networking. For example, it has been shown empirically
in \cite{ShTa} that the internet topology embeds with better accuracy
into a hyperbolic space than into an Euclidean space
of comparable dimension; the same holds for many complex networks, see \cite{KPKVB}.
A few algorithmic problems in
hyperbolic spaces and hyperbolic graphs have been considered
in recent papers (see \cite{ChEs,Epp,GaLy,Kra}).
Another important
application of these spaces is the study of the spread of viruses through on the
internet (see \cite{K21,K22}).
Furthermore, hyperbolic spaces are useful in secure transmission of information on the
network (see \cite{K27,K21,K22,NS}).

The study of Gromov hyperbolic graphs is a subject of increasing interest; see, \emph{e.g.}, \cite{BRS,BRSV2,BRST,BPK,BHB1,CDR,CPRS,CRS,CRSV,CDEHV,K50,K27,K21,K22,K23,K24,K56,KPKVB,MRSV,MRSV2,NS,PeRSV,PRST,PRSV,PT,R,RSVV,S,S2,T,WZ} and the references therein.

We say that the curve $\g$ in a metric space $X$ is a
\emph{geodesic} if we have $L(\g|_{[t,s]})=d(\g(t),\g(s))=|t-s|$ for every $s,t\in [a,b]$
(then $\gamma$ is equipped with an arc-length parametrization).
The metric space $X$ is said \emph{geodesic} if for every couple of points in
$X$ there exists a geodesic joining them; we denote by $[xy]$
any geodesic joining $x$ and $y$; this notation is ambiguous, since in general we do not have uniqueness of
geodesics, but it is very convenient.
Consequently, any geodesic metric space is connected.
If the metric space $X$ is
a graph, then the edge joining the vertices $u$ and $v$ will be denoted by $[u,v]$.

Along the paper we just consider graphs with every edge of length $1$.
In order to consider a graph $G$ as a geodesic metric space, identify (by an isometry)
any edge $[u,v]\in E(G)$ with the interval $[0,1]$ in the real line;
then the edge $[u,v]$ (considered as a graph with just one edge)
is isometric to the interval $[0,1]$.
Thus, the points in $G$ are the vertices and, also, the points in the interior
of any edge of $G$.
In this way, any connected graph $G$ has a natural distance
defined on its points, induced by taking shortest paths in $G$,
and we can see $G$ as a metric graph.
If $x,y$ are in different connected components of $G$, we define $d_G(x,y)=\infty$.
Throughout this paper, $G=(V,E)$ denotes a simple graph (not necessarily connected) such that every edge has length $1$
and $V\neq \emptyset$.
These properties guarantee that any connected graph is a geodesic metric space. 
Note that to exclude multiple
edges and loops is not an important loss of generality, since
\cite[Theorems 8 and 10]{BRSV2} reduce the problem of compute
the hyperbolicity constant of graphs with multiple edges and/or
loops to the study of simple graphs.
For a nonempty set $X\subseteq V$, and a vertex $v\in V$, $N_X(v)$ denotes the set of neighbors $v$ has in $X$:
$N_X(v):=\{u\in X: [u,v]\in E\},$ and the degree of $v$ in $X$ will be
denoted by $\deg_{X}(v)=|N_{X}(v)|$. We denote the degree of a vertex $v\in V$ in $G$  by $\deg(v)\le\infty$, and
the maximum degree of $G$ by  $\D_{G}:=\sup_{v\in V}\deg(v)$.

Consider a polygon $J=\{J_1,J_2,\dots,J_n\}$
with sides $J_j\subseteq X$ in a geodesic metric space $X$.
We say that $J$ is $\d$-{\it thin} if for
every $x\in J_i$ we have that $d(x,\cup_{j\neq i}J_{j})\le \d$.
Let us
denote by $\d(J)$ the sharp thin constant of $J$, \emph{i.e.},
$\d(J):=\inf\{\d\ge 0: \, J \, \text{ is $\d$-thin}\,\}\,. $
If $x_1,x_2,x_3$ are three points in $X$, a {\it geodesic triangle} $T=\{x_1,x_2,x_3\}$ is
the union of the three geodesics $[x_1x_2]$, $[x_2x_3]$ and
$[x_3x_1]$ in $X$.
We say that $X$ is $\d$-\emph{hyperbolic} if every geodesic
triangle in $X$ is $\d$-thin, and we denote by $\d(X)$ the sharp
hyperbolicity constant of $X$, \emph{i.e.}, $\d(X):=\sup\{\d(T): \, T \,
\text{ is a geodesic triangle in }\,X\,\}.$ We say that $X$ is
\emph{hyperbolic} if $X$ is $\d$-hyperbolic for some $\d \ge 0$; then $X$ is hyperbolic if and only if
$ \d(X)<\infty.$
If $X$ has connected components $\{X_i\}_{i\in I}$, then we define $\d(X):=\sup_{i\in I} \d(X_i)$, and we say that $X$ is hyperbolic if $\d(X)<\infty$.

In the classical references on this subject (see, \emph{e.g.}, \cite{BHB,GH})
appear several different definitions of Gromov hyperbolicity, which are equivalent in the sense
that if $X$ is $\d$-hyperbolic with respect to one definition,
then it is $\d'$-hyperbolic with respect to another definition (for some $\d'$ related to $\d$).
The definition that we have chosen has a deep geometric meaning (see, \emph{e.g.}, \cite{GH}).

Trivially, any bounded
metric space $X$ is $((\diam X)/2)$-hyperbolic.
A normed linear space is hyperbolic if and only if it has dimension one.
A geodesic space is $0$-hyperbolic if and only if it is a metric tree.
If a complete Riemannian manifold is simply connected and its sectional curvatures satisfy
$K\leq c$ for some negative constant $c$, then it is hyperbolic.
See the classical references \cite{ABCD,GH} in order to find further results.

We want to remark that the main examples of hyperbolic graphs are the trees.
In fact, the hyperbolicity constant of a geodesic metric space can be viewed as a measure of
how ``tree-like'' the space is, since those spaces $X$ with $\delta(X) = 0$ are precisely the metric trees.
This is an interesting subject since, in
many applications, one finds that the borderline between tractable and intractable
cases may be the tree-like degree of the structure to be dealt with
(see, \emph{e.g.}, \cite{CYY}).

Given a Cayley graph (of a presentation with solvable word problem)
there is an algorithm which allows to decide if it is hyperbolic.
However, for a general graph or a general geodesic metric space
deciding whether or not a space is
hyperbolic is usually very difficult.
Therefore, it is interesting to study the hyperbolicity of particular classes of graphs.
The papers \cite{BRST,BHB1,CCCR,CDR,CRSV,MRSV2,PeRSV,PRSV,R,Si} study the hyperbolicity of, respectively, complement of graphs, chordal
graphs, strong product graphs, lexicographic product graphs,
line graphs, Cartesian product graphs, cubic graphs, tessellation graphs, short graphs and median graphs.
In \cite{CCCR,CDR,MRSV2} the authors characterize the hyperbolic product graphs (for strong product, lexicographic product and Cartesian product) in terms of properties of the factor graphs.
In this paper we characterize the hyperbolic product graphs for graph join $G_1\uplus G_2$ and the corona $G_1\diamond G_2$:
$G_1\uplus G_2$ is always hyperbolic, and $G_1\diamond G_2$ is hyperbolic if and only if $G_1$ is hyperbolic
(see Corollaries \ref{cor:SP} and \ref{cor:sup}).
Furthermore, we obtain simple formulae for the hyperbolicity constant of the graph join $G_1\uplus G_2$ and the corona $G_1\diamond G_2$ (see Theorems \ref{th:hypJoin} and \ref{th:corona}).
In particular, Theorem \ref{th:corona} states that
$\d(G_1\diamond G_2)=\max\{\d(G_1),\d(G_2\uplus E_1)\}$, where $E_1$ is a graph with just one vertex.
We want to remark that it is not usual at all to obtain explicit formulae for the hyperbolicity constant of large classes of graphs.

\section{Distance in graph join}

In order to estimate the hyperbolicity constant of the graph join $G_1\uplus G_2$ of $G_1$ and $G_2$, we will need an explicit formula for the distance between two arbitrary points.
We will use the definition given by Harary in \cite{H}.

\begin{definition}\label{def:join}
Let $G_1=(V(G_1),E(G_1))$ and $G_2=(V(G_2),E(G_2))$ two graphs with $V(G_1)\cap V(G_2)=\varnothing$. The \emph{graph join} $G_1\uplus G_2$ of $G_1$ and $G_2$ has $V(G_1\uplus G_2)=V(G_1) \cup V(G_2)$ and two different vertices $u$ and $v$ of $G_1\uplus G_2$ are adjacent if $u\in V(G_1)$ and $v\in V(G_2)$, or $[u,v]\in E(G_1)$ or $[u,v]\in E(G_2)$.
\end{definition}

From the definition, it follows that the graph join of two graphs is commutative.
Figure \ref{fig:join} shows the graph join of two graphs.

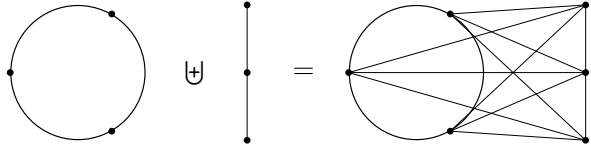
\begin{figure}[h]
\centering
\scalebox{.9}
{\begin{pspicture}(-1.2,-1.2)(7.7,1.2)
\pscircle[linewidth=.5pt](0,0){1}
\cnode*[](-1,0){0.05}{A}
\cnode*[](0.5,0.866025){0.05}{B}
\cnode*[](0.5,-0.866025){0.05}{C}
\cnode*[](2.5,1){0.05}{E}
\cnode*[](2.5,0){0.05}{F}
\cnode*[](2.5,-1){0.05}{G}
\psline[linewidth=0.01cm,dotsize=0.07055555cm 2.5]{-}(2.5,1)(2.5,-1)
\pscircle[linewidth=.5pt](5,0){1}
\cnode*[](4,0){0.05}{A'}
\cnode*[](5.5,0.866025){0.05}{B'}
\cnode*[](5.5,-0.866025){0.05}{C'}
\cnode*[](7.5,1){0.05}{E'}
\cnode*[](7.5,0){0.05}{F'}
\cnode*[](7.5,-1){0.05}{G'}
\psline[linewidth=0.01cm,dotsize=0.07055555cm 2.5]{-}(7.5,1)(7.5,-1)
\psline[linewidth=0.01cm,dotsize=0.07055555cm 2.5]{-}(4,0)(7.5,1)
\psline[linewidth=0.01cm,dotsize=0.07055555cm 2.5]{-}(4,0)(7.5,0)
\psline[linewidth=0.01cm,dotsize=0.07055555cm 2.5]{-}(4,0)(7.5,-1)
\psline[linewidth=0.01cm,dotsize=0.07055555cm 2.5]{-}(5.5,0.866025)(7.5,1)
\psline[linewidth=0.01cm,dotsize=0.07055555cm 2.5]{-}(5.5,0.866025)(7.5,0)
\psline[linewidth=0.01cm,dotsize=0.07055555cm 2.5]{-}(5.5,0.866025)(7.5,-1)
\psline[linewidth=0.01cm,dotsize=0.07055555cm 2.5]{-}(5.5,-0.866025)(7.5,1)
\psline[linewidth=0.01cm,dotsize=0.07055555cm 2.5]{-}(5.5,-0.866025)(7.5,0)
\psline[linewidth=0.01cm,dotsize=0.07055555cm 2.5]{-}(5.5,-0.866025)(7.5,-1)
\uput[0](1.4,0){$\biguplus$}
\uput[0](3,0){\large{=}}
\end{pspicture}}
\caption{Graph join of two graphs $C_3 \uplus P_3$.} \label{fig:join}
\end{figure}

\begin{remark}\label{r:K_nm}
For every graphs $G_1,G_2$ we have that $G_1\uplus G_2$ is a connected graph with a subgraph isomorphic to a complete bipartite graph with $V(G_1)$ and $V(G_2)$ as its parts.
\end{remark}

Note that, from a geometric viewpoint, the graph join $G_1\uplus G_2$ is obtained as an union of the graphs $G_1$, $G_2$ and the complete bipartite graph $K(G_1,G_2)$ linking the vertices of $V(G_1)$ and $V(G_2)$.

The following result allows to compute the distance between any two points in $G_1\uplus G_2$. Furthermore, this result provides information about the geodesics in the graph join.

\begin{proposition}\label{prop:JoinDist}
For every graphs $G_1, G_2$ we have:

\begin{itemize}
\item[(a)] If $x,y \in G_i$ ($i\in\{1,2\}$), then
             \[d_{G_1\uplus G_2}(x,y) = \min\left\{ d_{G_i}(x,y) , d_{G_i}\big(x,V(G_i)\big)+2+d_{G_i}\big(V(G_i),y\big)\right\}.\]
\item[(b)] If $x \in G_i$ and $y \in G_j$ with $i\neq j$, then
             \[d_{G_1\uplus G_2}(x,y) = d_{G_i}\big(x,V(G_i)\big)+1+d_{G_j}\big(V(G_j),y\big).\]
\item[(c)] If $x \in G_i$ and $y \in K(G_1,G_2)$, then
             \[d_{G_1\uplus G_2}(x,y) = \min\left\{ d_{G_i}(x,Y_i)+d_{G_1\uplus G_2}(Y_i,y) , d_{G_i}\big(x,V(G_i)\big)+1+d_{G_1\uplus G_2}(Y_j,y)\right\},\]
           where $y\in [Y_1,Y_2]$ with $Y_i\in V(G_i)$ and $Y_j\in V(G_j)$.
\item[(d)] If $x,y \in K(G_1,G_2)$, then
             \[d_{G_1\uplus G_2}(x,y) = \min\{ d_{K(G_1,G_2)}(x,y), M\},\]
           where $x\in [X_1,X_2]$, $y\in [Y_1,Y_2]$ with $X_1,Y_1\in V(G_1)$ and $X_2,Y_2\in V(G_2)$, and $M=\min_{i\in\{1,2\}}\{d_{G_1\uplus G_2}(x,X_i)+d_{G_i}(X_i,Y_i)+d_{G_1\uplus G_2}(Y_i,y)\}$ 
\end{itemize}
\end{proposition}

\begin{proof}
We will prove each item separately.
In item (a), if $i\neq j$, we consider the two shortest possible paths from $x$ to $y$ such that they either is contained in $G_i$ or intersects $G_j$ (and then it intersects $G_j$ just in a single vertex).
In item (b), since any path in $G_1\uplus G_2$ joining $x$ and $y$ contains at less one edge in $K(G_1,G_2)$, we have a geodesic when the path contains an edge joining a closest vertex to $x$ in $V(G_i)$ and a closest vertex to $y$ in $V(G_j)$.
In item (c) we consider the two shortest possible paths from $x$ to $y$ containing either $Y_1$ or $Y_2$.
Finally, in item (d) we may consider the three shortest possible paths from $x$ to $y$ such that they either is contained in $K(G_1,G_2)$ or contains at lest an edge in $E(G_1)$ or contains at lest an edge in $E(G_2)$.
\end{proof}

We say that a subgraph  $\G$ of $G$ is \emph{isometric} if $d_{\G}(x,y)=d_{G}(x,y)$ for every $x,y\in \G$. Proposition \ref{prop:JoinDist} gives the following result.

\begin{proposition}\label{prop:IsomJoin}
  Let $G_1,G_2$ be two graph and let $\G_1,\G_2$ be isometric subgraphs to $G_1$ and $G_2$, respectively. Then, $\G_1\uplus\G_2$ is an isometric subgraph to $G_1\uplus G_2$.
\end{proposition}

The following result allows to compute the diameter of the set of vertices in a graph join.

\begin{proposition}\label{prop:vert}
For every graphs $G_1,G_2$ we have $1\le\diam V(G_1\uplus G_2)\le 2$. Furthermore, $\diam V(G_1\uplus G_2)=1$ if and only if $G_1$ and $G_2$ are complete graphs.
\end{proposition}

\begin{proof}
Since $V(G_1),V(G_2)\neq\emptyset$, $\diam V(G_1\uplus G_2)\ge 1$. Besides, if $u,v\in V(G_1\uplus G_2)$, we have $d_{G_1\uplus G_2}(u,v)\le d_{K(G_1,G_2)}(u,v)\le 2$.

In order to finish the proof note that on the one hand, if $G_1$ and $G_2$ are complete graphs, then $G_1\uplus G_2$ is a complete graph with at least $2$ vertices and $\diam V(G_1\uplus G_2)=1$. On the other hand, if $\diam V(G_1\uplus G_2)=1$, then for every two vertices $u,v \in V(G_1)$ we have $[u,v]\in E(G_1)$; by symmetry, we have the same result for every $u,v \in V(G_2)$.
\end{proof}

Since $\diam V(G) \le \diam G \le \diam V(G) + 1$ for every graph $G$, the previous proposition has the following consequence.

\begin{corollary}\label{c:diam}
For every graphs $G_1,G_2$ we have $1\le\diam G_1\uplus G_2\le 3$.
\end{corollary}

Proposition \ref{prop:JoinDist} and Corollary \ref{c:diam} give the following results.
Given a graph $G$, we say that $x\in G$ is a midpoint (of an edge) if $d_{G}(x,V(G))=1/2$.

\begin{corollary}\label{cor:midpoint}
  Let $G_1,G_2$ be two graphs. If $d_{G_1\uplus G_2}(x,y) = 3$, then $x,y$ are two midpoints in $G_i$ with $d_{G_i}(x,y)\ge3$ for some $i\in \{1,2\}$.
\end{corollary}

\begin{corollary}\label{r:diam3}
  Let $G_1,G_2$ be two graphs. Then, $\diam G_1\uplus G_2 = 3$ if and only if there are two midpoints $x,y$ in $G_i$ with $d_{G_i}(x,y)\ge3$ for some $i\in \{1,2\}$.
\end{corollary}

\section{Hyperbolicity constant of the graph join of two graphs}

In this section we obtain some bounds for the hyperbolicity constant of the graph join of two graphs.
These bounds allow to prove that the joins of graphs are always hyperbolic with a small hyperbolicity constant.
The next well-known result will be useful.

\begin{theorem}\cite[Theorem 8]{RSVV}\label{t:diameter1}
In any graph $G$ the inequality $\d(G)\le \diam G / 2$ holds and it is sharp.
\end{theorem}

We have the following consequence of Corollary \ref{c:diam} and Theorem \ref{t:diameter1}.

\begin{corollary}\label{cor:SP}
For every graphs $G_1,G_2$, the graph join $G_1\uplus G_2$ is hyperbolic with $\d(G_1\uplus G_2)\leq 3/2$, and the inequality is sharp.
\end{corollary}

Theorem \ref{th:hyp3/2} characterizes the graph join of two graphs for which the equality in the previous corollary is attained.

The following result in \cite[Lemma 5]{RSVV} will be useful.

\begin{lemma}\label{l:subgraph}
If $\G$ is an isometric subgraph of $G$, then $\d(\G) \le \d(G)$.
\end{lemma}

\begin{theorem}\label{th:HypIsomJoin}
For every graphs $G_1,G_2$, we have $$\d(G_1\uplus G_2)=\max\{ \d(\G_1\uplus \G_2) : \G_i \text{ is isometric to } G_i \text{ for } i=1,2 \}.$$
\end{theorem}

\begin{proof}
By Proposition \ref{prop:IsomJoin} and Lemma \ref{l:subgraph} we have $\d(G_1\uplus G_2)\ge \d(\G_1\uplus \G_2)$ for any isometric subgraph $\G_i$ of $G_i$ for $i=1,2$. Besides, since any graph is an isometric subgraph of itself we obtain the equality by taking $\G_1=G_1$ and $\G_2=G_2$.
\end{proof}

Denote by $J(G)$ the set of vertices and midpoints of edges in $G$.
As usual, by \emph{cycle} we mean a simple closed curve, i.e., a path with different vertices,
unless the last one, which is equal to the first vertex.

First, we collect some previous results of \cite{BRS} which will be useful.

\begin{theorem}\cite[Theorem 2.6]{BRS}
\label{t:multk/4}
For every hyperbolic graph $G$, $\d(G)$ is a multiple of $1/4$.
\end{theorem}

\begin{theorem}\cite[Theorem 2.7]{BRS}
\label{t:TrianVMp}
For any hyperbolic graph $G$, there exists
a geodesic triangle $T = \{x, y, z\}$ that is a cycle with $x, y, z \in J(G)$ and $\d(T) = \d(G)$.
\end{theorem}

The following result characterizes the hyperbolic graphs with a small hyperbolicity constant, see \cite[Theorem 11]{MRSV}.
Let us define the \emph{circumference} $c(G)$ of a graph $G$ which is not a tree as the supremum of the lengths of its cycles; if $G$ is a tree we define $c(G)=0$.

\begin{theorem}\label{th:delt<1}
Let $G$ be any graph.
\begin{itemize}
  \item[(a)] {$\d(G) = 0$ if and only if $G$ is a tree.}
  \item[(b)] {$\d(G) = 1/4, 1/2$ is not satisfied for any graph $G$.}
  \item[(c)] {$\d(G) = 3/4$ if and only if $\ c(G)=3$.}
\end{itemize}
\end{theorem}

\medskip

We have the following consequence for the hyperbolicity constant of the joins of graphs.

\begin{proposition}\label{r:discretJoin}
  For every graphs $G_1,G_2$ the graph join $G_1\uplus G_2$ is hyperbolic with hyperbolicity constant $\d(G_1\uplus G_2)$ in $\{0, 3/4, 1, 5/4, 3/2\}$.
\end{proposition}

If $G_1$ and $G_2$ are \emph{isomorphic}, then we write $G_1 \simeq G_2$. It is clear that if $G_1\simeq G_2$, then $\d(G_1)=\d(G_2)$.

The $n$-vertex edgeless graph ($n\ge1$) or \emph{empty graph} is a graph without edges and with $n$ vertices, and it is commonly denoted as $E_n$.

The following result allows to characterize the joins of graphs with hyperbolicity constant less than one in terms of its factor graphs.
Recall that $\D_G$ denotes the maximum degree of the vertices in $G$.

\begin{theorem}\label{th:deltJoin<1}
 Let $G_1,G_2$ be two graphs.
 \begin{itemize}
   \item[(1)] {$\d(G_1\uplus G_2)=0$ if and only if $G_1$ and $G_2$ are empty graphs and one of them is isomorphic to $E_1$.}
   \item[(2)] {$\d(G_1\uplus G_2)=3/4$ if and only if $G_1\simeq E_1$ and $\D_{G_2}=1$, or $G_2\simeq E_1$ and $\D_{G_1}=1$.}
 \end{itemize}
\end{theorem}

\begin{proof}$ $
\begin{itemize}
  \item[(1)] {By Theorem \ref{th:delt<1} it suffices to characterize the joins of graphs which are trees. If $G_1$ and $G_2$ are empty graphs and one of them is isomorphic to $E_1$, then it is clear that $G_1\uplus G_2$ is a tree. Assume now that $G_1\uplus G_2$ is a tree. If $G_1$ and $G_2$ have at least two vertices then $G_1\uplus G_2$ has a cycle with length four. Thus, $G_1$ or $G_2$ is isomorphic to $E_1$. Without loss of generality we can assume that $G_1\simeq E_1$. Note that if $G_2$ has at least one edge then $G_1\uplus G_2$ has a cycle with length three. Then, $G_2\simeq E_n$ for some $n\in \mathbb{N}$.}
  \item[(2)] {By Theorem \ref{th:delt<1} it suffices to characterize the joins of graphs with circumference three. If $G_1\simeq E_1$ and $\D_{G_2}=1$, or $G_2\simeq E_1$ and $\D_{G_1}=1$, then it is clear that $c(G_1\uplus G_2)=3$. Assume now that $c(G_1\uplus G_2)=3$. If $G_1,G_2$ both have at least two vertices then $G_1\uplus G_2$ contains a cycle with length four and so $c(G_1\uplus G_2)\ge4$. Therefore, $G_1$ or $G_2$ is isomorphic to $E_1$. Without loss of generality we can assume that $G_1\simeq E_1$. Note that if $\D_{G_2}\ge2$ then there is an isomorphic subgraph to $E_1\uplus P_3$ in $G_1\uplus G_2$; thus, $G_1\uplus G_2$ contains a cycle with length four. So, we have $\D_{G_2}\le1$. Besides, since $G_2$ is a non-empty graph by (1), we have $\D_{G_2}\ge1$.}
\end{itemize}
\end{proof}

The following result will be useful, see \cite[Theorem 11]{RSVV}.
The graph join of a cycle $C_{n-1}$ and a single vertex $E_1$ is referred to as a \emph{wheel} with $n$ vertices and denoted by $W_n$. Notice that the complete bipartite graph $K_{n,m}$ is isomorphic to the graph join of two empty graphs $E_n,E_m$, i.e., $K_{n,m}\simeq E_n\uplus E_m$.

\begin{example}\label{examples}
The following graphs have these hyperbolicity constants:
\begin{itemize}
  \item The wheel graph with $n$ vertices $W_n$ verifies $\d(W_4)=\d(W_5)=1$, $\d(W_n)=3/2$ for every $7\le n\le 10$, and $\d(W_n)=5/4$ for $n=6$ and for every $n\ge 11$.
  \item The complete bipartite graphs verify $\d(K_{1,n}) = 0$ for every $n\ge1$, $\d(K_{m,n}) = 1$ for every $m,n \ge2$.
\end{itemize}
\end{example}

Theorem \ref{th:deltJoin<1} and Example \ref{examples} show that the family of graphs $E_1\uplus G$ when $G$ belongs to the set of graphs is a representative collection of joins of graphs since their hyperbolicity constants take all possible values.

The following results characterize the graphs with hyperbolicity constant one and greater than one, respectively. If $G_0$ is a subgraph of $G$ and $w\in V(G_0)$, we denote by $\deg_{G_0}(w)$ the degree of $w$ in the induced subgraph by $V(G_0)$.

\begin{theorem}\cite[Theorem 3.10]{BRS2}\label{th:delt=1}
Let $G$ be any graph. Then $\d(G) = 1$ if and only if the following conditions hold:
\begin{itemize}
  \item[(1)] {There exists a cycle isomorphic to $C_4$.}
  \item[(2)] {For every cycle $\s$ with $L(\s) \ge 5$ and for every vertex $w \in \s$, we have $\deg_\s(w) \ge3$.}
\end{itemize}
\end{theorem}


\begin{theorem}\cite[Theorem 3.2]{BRS2}\label{th:delt>=5/4}
Let $G$ be any graph. Then $\d(G) \ge 5/4$ if and only if there exist a cycle $\s$ in $G$ with length $L(\s) \ge 5$ and a vertex $w \in V(\s)$ such that $\deg_\s(w) = 2$.
\end{theorem}

Theorem \ref{th:delt>=5/4} has the following consequence for joins of graphs.

\begin{lemma}\label{l:Fact_Delt>1}
Let $G_1,G_2$ be two graphs. If $\d(G_1)>1$, then $\d(G_1\uplus G_2)>1$.
\end{lemma}

\begin{proof}
 By Theorem \ref{th:delt>=5/4}, there exist a cycle $\s$ in $G_1\uplus G_2$ (contained in $G_1$) with length $L(\s) \ge 5$ and a vertex $w \in\s$ such that $\deg_\s(w) = 2$. Thus, Theorem \ref{th:delt>=5/4} gives $\d(G_1\uplus G_2)>1$.
\end{proof}

Note that the converse of Lemma \ref{l:Fact_Delt>1} does not hold, since $\d(E_1)=\d(P_4)=0$ and we can check that $\d(E_1\uplus P_4)=5/4$.

\begin{corollary}\label{c:FactDelt>1}
Let $G_1,G_2$ be two graphs. Then $$\d(G_1\uplus G_2)\ge \min\big\{5/4,\max\{\d(G_1),\d(G_2)\}\big\}.$$
\end{corollary}

\begin{proof}
By symmetry, it suffices to show $\d(G_1\uplus G_2)\ge \min\{5/4,\d(G_1)\}$.
If $\d(G_1)>1$, then the inequality holds by Lemma \ref{l:Fact_Delt>1}.
If $\d(G_1)=1$, then there exists a cycle isomorphic to $C_4$ in $G_1\subset G_1\uplus G_2$; hence, $\d(G_1\uplus G_2)\ge1$.
If $\d(G_1)=3/4$, then there exists a cycle isomorphic to $C_3$ in $G_1\subset G_1\uplus G_2$; hence, $\d(G_1\uplus G_2)\ge3/4$.
The inequality is direct if $\d(G_1)=0$.
\end{proof}

The following results allow to characterize the joins of graphs with hyperbolicity constant one in terms of $G_1$ and $G_2$.

\begin{lemma}\label{l:EmptyJoin}
 Let $G$ be any graph. Then, $\d(E_1\uplus G)\le1$ if and only if every path $\eta$ joining two vertices of $G$ with $L(\eta) = 3$ satisfies $\deg_\eta(w)\ge2$ for every vertex $w \in V(\eta)$.
\end{lemma}

Note that if every path $\eta$ joining two vertices of $G$ with $L(\eta) = 3$ satisfies $\deg_\eta(w)$ $\ge2$ for every vertex $w \in V(\eta)$, then the same result holds for $L(\eta)\ge3$ instead of $L(\eta)=3$.

\begin{proof}
Let $v$ be the vertex in $E_1$.

Assume first that $\d(E_1\uplus G)\le1$. Seeking for a contradiction, assume that there is a path $\eta$ joining two vertices of $G$ with $L(\eta) = 3$ and one vertex $w' \in V(\eta)$ with $\deg_\eta(w')=1$. Consider now the cycle $\s$ obtained by joining the endpoints of $\eta$ with $v$. Note that $w'\in \s$ and $\deg_\s(w')=2$; therefore, Theorem \ref{th:delt>=5/4} gives $\d(E_1\uplus G)>1$, which is a contradiction.

Assume now that every path $\eta$ joining two vertices of $G$ with $L(\eta) = 3$ satisfies $\deg_\eta(w)\ge2$ for every vertex $w \in V(\eta)$.
Note that if $G$ does not have paths isomorphic to $P_4$ then there is no cycle in $E_1\uplus G$ with length greater than $4$ and so, $\d(E_1\uplus G)\le1$.
We are going to prove now that for every cycle $\s$ in $G$ with $L(\s) \ge 5$ we have $\deg_{\s}(w)\ge3$ for every vertex $w \in V(\s)$.
Let $\s$ be any cycle in $E_1\uplus G$ with $L(\s) \ge 5$.
If $v\in \s$, then $\s\cap G$ is a subgraph of $G$ isomorphic to $P_{n}$ for $n=L(\s)-1$, and $\deg_\s(v)=n\ge4$.
Since $L(\s\cup G)\ge3$, $\deg_{\s\cap G}(w)\ge2$ for every $w\in V(\s\cap G)$ by hypothesis, and we conclude $\deg_\s(w)\ge3$ for every $w\in V(\s)\setminus\{v\}$.
If $v\notin \s$, let $w$ be any vertex in $\s$ and let $P(w)$ be a path with length $3$ contained in $\s$ and such that $w$ is an endpoint of $P(w)$. By hypothesis $\deg_{P(w)}(w)\ge2$; since $w$ has a neighbor $w'\in V(\s\setminus P(w))$, $\deg_{\s}(w)\ge3$ for any $w\in V(\s)$.
Then, Theorem \ref{th:delt>=5/4} gives the result.
\end{proof}

Note that if a graph $G$ verifies $\diam G\le2$ then every path $\eta$ joining two vertices of $G$ with $L(\eta) = 3$ satisfies $\deg_\eta(w)\ge2$ for every vertex $w \in V(\eta)$.  The converse does not hold, since in the disjoint union $C_3\cup C_3$ of two cycles $C_3$ any path with length $3$ is a cycle and $\diam C_3\cup C_3 = \infty$. However, these two conditions are equivalent if $G$ is connected.

If $G$ is a graph with connected components $\{G_j\}$, we define
$$
\diam^* G :=\sup_{j} \, \diam G_j.
$$
Note that $\diam^* G = \diam G$ if $G$ is connected; otherwise, $\diam G=\infty$.
Also, $\diam^* G$ $>1$ is equivalent to $\D_{G}\ge2$.
We also have the following result:

\begin{lemma}\label{lemaX}
Let $G$ be any graph. Then $\diam^* G\le2$ if and only if every $\eta$ joining two vertices of $G$ with $L(\eta)=3$ satisfy $\deg_\eta(w)\ge2$ for every $w\in V(G)$.
\end{lemma}

\begin{lemma}\label{l:Fact2Vert}
Let $G_1$ and $G_2$ be two graphs with at least two vertices. Then, $\d(G_1\uplus G_2)=1$ if and only if $\diam G_i\le2$ or $G_i$ is an empty graph for $i=1,2$.
\end{lemma}

\begin{proof}
Assume that $\d(G_1\uplus G_2)=1$. Seeking for a contradiction, assume that $\diam G_1$ $\ge5/2$ and $G_1$ is a non-empty graph or $\diam G_2 \ge 5/2$ and $G_2$ is a non-empty graph. By symmetry, without loss of generality we can assume that $\diam G_1\ge5/2$ and $G_1$ is a non-empty graph; hence, there are a vertex $v\in V(G_1)$ and a midpoint $p\in [w_1,w_2]$ with $d_{G_1}(v,p)\ge5/2$. Consider a cycle $\s$ in $G_1\uplus G_2$ containing the vertex $v$, the edge $[w_1,w_2]$ and two vertices of $G_2$, with $L(\s)=5$. We have $\deg_{\s}(v)=2$. Thus, Theorem \ref{th:delt>=5/4} gives $\d(G_1\uplus G_2)>1$. This contradicts our assumption, and so, we obtain $\diam G_1\le2$.

\medskip

Assume now that $\diam G_i\le2$ or $G_i$ is an empty graph for $i=1,2$.
Since $G_1$ and $G_2$ have at least two vertices, there exists a cycle isomorphic to $C_4$ in $G_1\uplus G_2$.

First of all, if $G_1$ and $G_2$ are empty graphs then Example \ref{examples} gives $\d(G_1\uplus G_2)=1$.

Without loss of generality we can assume that $G_1$ is a non-empty graph, then $G_1$ satisfies $\diam G_1\le2$.

Assume that $G_2$ is an empty graph. Let $\s$ be any cycle in $G_1\uplus G_2$ with $L(\s)\ge5$. Since $\s$ contains at least three vertices in $G_1$, we have $\deg_{\s}(v)=|V(G_1)\cap\s|\ge3$ for every $v\in V(G_2)\cap\s$.
Besides, if $|V(G_2)\cap\s|\ge3$ then $\deg_{\s}(w)\ge|V(G_2)\cap\s|\ge3$ for every $w\in V(G_1)\cap\s$. If $|V(G_2)\cap\s|=1$, then $\eta:=\s\cap G_1$ is a path in $G_1$ with $L(\eta)\ge3$, and so, $\deg_\eta(w)\ge2$ and $\deg_\s(w)\ge3$ for every $w \in V(\eta)$.
If $|V(G_2)\cap\s|=2$, then $\s\cap G_1$ is the union of two paths and $|V(G_1)\cap\s|\ge3$; since $\diam G_1\le2$, we have $\deg_{G_1\cap\s}(w)\ge1$ for every $w\in V(G_1)\cap\s$ (otherwise there are a vertex $w\in V(G_1)\cap\s$ and a midpoint $p\in G_1\cap \s$ with $d_{G_1}(w,p)>2$).
Then, we have $\deg_{\s}(v)\ge3$ for every $v\in V(\s)$ and so, we obtain $\d(G_1\uplus G_2)=1$ by Theorem \ref{th:delt=1}.

\smallskip

Finally, assume that $\diam G_2\le2$.
By Theorem \ref{t:TrianVMp} it suffices to consider geodesic triangles $T=\{x,y,z\}$ in $G_1\uplus G_2$ that are cycles with $x,y,z\in J(G_1\uplus G_2)$. So, since $\diam G_1,\diam G_2\le2$, Proposition \ref{prop:JoinDist} gives that $L([xy]),L([yz]),L([zx]) \le 2$; thus, for every $\a\in[xy]$, $d_{G_1\uplus G_2}(\a,[yz]\cup[zx])\le d_{G_1\uplus G_2}(\a,\{x,y\})\le L([xy])/2$. Hence, $\d(T)\le \max\{L([xy]),L([yz]),L([zx])\}/2\le1$ and so, $\d(G_1\uplus G_2)\le1$. Since $G_1$ and $G_2$ have at least two vertices, by Theorem \ref{th:delt<1} we have $\d(G_1\uplus G_2)\ge1$ and we conclude $\d(G_1\uplus G_2)=1$.
\end{proof}

The following result characterizes the joins of graphs with hyperbolicity constant one.

\begin{theorem}\label{th:hypJoin1}
Let $G_1,G_2$ be any two graphs. Then the following statements hold:
\begin{itemize}
  \item {Assume that $G_1\simeq E_1$. Then $\d(G_1\uplus G_2)=1$ if and only if $1<\diam^* G_2\le2$.}
  \item {Assume that $G_1$ and $G_2$ have at least two vertices. Then $\d(G_1\uplus G_2)=1$ if and only if $\diam G_i\le2$ or $G_i$ is an empty graph for $i=1,2$.}
\end{itemize}
\end{theorem}

\begin{proof}
We have the first statement by Theorem \ref{th:deltJoin<1} and Lemmas \ref{l:EmptyJoin} and \ref{lemaX}.
The second statement is just Lemma \ref{l:Fact2Vert}.
\end{proof}

In order to compute the hyperbolicity constant of any graph join we are going to characterize the joins of graphs with hyperbolicity constant $3/2$.

\begin{lemma}\label{l:hyp3/2Fact}
Let $G_1,G_2$ be any two graphs. If $\d(G_1\uplus G_2)=3/2$, then each geodesic triangle $T=\{x,y,z\}$ in $G_1\uplus G_2$ that is a cycle with $x,y,z \in J(G_1\uplus G_2)$ and $\d(T)=3/2$ is contained in either $G_1$ or $G_2$.
\end{lemma}

\begin{proof}
Seeking for a contradiction assume that there is a geodesic triangle $T=\{x,y,z\}$ in $G_1\uplus G_2$ that is a cycle with $x,y,z \in J(G_1\uplus G_2)$ and $\d(T)=3/2$ which contains vertices in both factors $G_1,G_2$.
Without loss of generality we can assume that there is $p\in[xy]$ with $d_{G_1\uplus G_2}(p, [yz]\cup[zx]) = 3/2$, and so, $L([xy])\ge3$. Hence, $d_{G_1\uplus G_2}(x,y)=3$ by Corollary \ref{c:diam}, and by Corollary \ref{cor:midpoint} we have that $x,y$ are midpoints either in $G_1$ or in $G_2$, and so, $p$ is a vertex in $G_1\uplus G_2$. Without loss of generality we can assume that $x,y\in G_1$.
Let $V_x$ be the closest vertex to $x$ in $[xz]\cup[zy]$.
If $p\in V(G_2)$ then $d_{G_1\uplus G_2}(p,[yz]\cup[zx]) \le d_{G_1\uplus G_2}(p,V_x) =1$. This contradicts our assumption.
If $p\in V(G_1)$ then since $T$ contains vertices in both factors, we have $d_{G_1\uplus G_2}(p,[yz]\cup[zx]) \le d_{G_1\uplus G_2}\big((p,V(G_2)\cap\{[yz]\cup[zx]\}\big) =1$. This also contradicts our assumption, and so, we have the result.
\end{proof}

\begin{corollary}\label{c:3/2}
Let $G_1,G_2$ be any two graphs. If $\d(G_1\uplus G_2)=3/2$, then $$\max\{\d(G_1),\d(G_2)\}\ge3/2.$$
\end{corollary}

\smallskip

The following families of graphs allow to characterize the joins of graphs with hyperbolicity constant $3/2$.
Denote by $C_n$ the cycle graph with $n\ge3$ vertices and by $V(C_n):=\{v_1^{(n)},\ldots,v_n^{(n)}\}$ the set of their vertices such that $[v_n^{(n)},v_1^{(n)}]\in E(C_n)$ and $[v_i^{(n)},v_{i+1}^{(n)}]\in E(C_n)$ for $1\le i\le n-1$.
Let us consider $\mathcal{C}_6^{(1)}$ the set of graphs obtained from $C_6$ by addying a (proper or not) subset of the set of edges $\{[v_2^{(6)},v_6^{(6)}]$, $[v_4^{(6)},v_6^{(6)}]\}$.
Let us define the set of graphs
$$\mathcal{F}_6:=\{ G \ \text{\small containing, as induced subgraph, an isomorphic graph to some element of } \mathcal{C}_6^{(1)}\}.$$
Let us consider $\mathcal{C}_7^{(1)}$ the set of graphs obtained from $C_7$ by addying a (proper or not) subset of the set of edges $\{[v_2^{(7)},v_6^{(7)}]$, $[v_2^{(7)},v_7^{(7)}]$, $[v_4^{(7)},v_6^{(7)}]$, $[v_4^{(7)},v_7^{(7)}]\}$.
Define
$$\mathcal{F}_7:=\{ G \ \text{\small containing, as induced subgraph, an isomorphic graph to some element of } \mathcal{C}_7^{(1)}\}.$$
Let us consider $\mathcal{C}_8^{(1)}$ the set of graphs obtained from $C_8$ by addying a (proper or not) subset of the set $\{[v_2^{(8)},v_6^{(8)}]$, $[v_2^{(8)},v_8^{(8)}]$, $[v_4^{(8)},v_6^{(8)}]$, $[v_4^{(8)},v_8^{(8)}]\}$.
Also, consider $\mathcal{C}_8^{(2)}$ the set of graphs obtained from $C_8$ by addying a (proper or not) subset of $\{[v_2^{(8)},v_8^{(8)}]$, $[v_4^{(8)},v_6^{(8)}]$, $[v_4^{(8)},v_7^{(8)}]$, $[v_4^{(8)},v_8^{(8)}]\}$.
Define
$$\mathcal{F}_8:=\{ G \ \text{\small containing, as induced subgraph, an isomorphic graph to some element of } \mathcal{C}_8^{(1)}\cup \mathcal{C}_8^{(2)}\}.$$
Let us consider $\mathcal{C}_9^{(1)}$ the set of graphs obtained from $C_9$ by addying a (proper or not) subset of the set of edges $\{[v_2^{(9)},v_6^{(9)}]$, $[v_2^{(9)},v_9^{(9)}]$, $[v_4^{(9)},v_6^{(9)}]$, $[v_4^{(9)},v_9^{(9)}]\}$.
Define
$$\mathcal{F}_9:=\{ G \ \text{\small containing, as induced subgraph, an isomorphic graph to some element of } \mathcal{C}_9^{(1)}\}.$$
Finally, we define the set $\mathcal{F}$ by
$$\mathcal{F}:=\mathcal{F}_6\cup\mathcal{F}_7\cup\mathcal{F}_8\cup\mathcal{F}_9.$$
Note that $\mathcal{F}_6$, $\mathcal{F}_7$, $\mathcal{F}_8$ and $\mathcal{F}_9$ are not disjoint sets of graphs.

\medskip

The following theorem characterizes the joins of graphs $G_1$ and $G_2$ with $\d(G_1\uplus G_2)=3/2$.
For any non-empty set $S\subset V(G)$, the induced subgraph of $S$ will be denoted by $\langle S\rangle$.

\begin{theorem}\label{th:hyp3/2}
Let $G_1,G_2$ be any two graphs. Then, $\d(G_1\uplus G_2)=3/2$ if and only if $G_1\in \mathcal{F}$ or $G_2\in\mathcal{F}$.
\end{theorem}

\begin{proof}
Assume first that $\d(G_1\uplus G_2)=3/2$.
By Theorem \ref{t:TrianVMp} there is a geodesic triangle $T=\{x,y,z\}$ in $G_1\uplus G_2$ that is a cycle with $x,y,z \in J(G_1)$ and $\d(T)=3/2$. By Lemma \ref{l:hyp3/2Fact}, $T$ is contained either in $G_1$ or in $G_2$.
Without loss of generality we can assume that $T$ is contained in $G_1$.
Without loss of generality we can assume that there is $p\in[xy]$ with $d_{G_1\uplus G_2}(p, [yz]\cup[zx]) = 3/2$, and by Corollary \ref{c:diam}, $L([xy])=3$.
Hence, by Corollary \ref{cor:midpoint} we have that $x,y$ are midpoints in $G_1$, and so, $p\in V(G_1)$.
Since $L([yz])\le3$, $L([zx])\le3$ and $L([yz])+L([zx])\ge L([xy])$, we have $6\le L(T)\le9$.

\smallskip

Assume that $L(T)=6$.
Denote by $\{v_1,\ldots,v_6\}$ the vertices in $T$ such that $T=\bigcup_{i=1}^{6}[v_i,v_{i+1}]$ with $v_7:=v_1$. Without loss of generality we can assume that $x\in[v_1,v_2]$, $y\in[v_4,v_5]$ and $p=v_3$. Since $d_{G_1\uplus G_2}(x,y)=3$, we have that $\langle\{v_1,\ldots,v_6\}\rangle$ contains neither $[v_1,v_4]$, $[v_1,v_5]$, $[v_2,v_4]$ nor $[v_2,v_5]$; besides, since $d_{G_1\uplus G_2}(p,[yz]\cup[zx])>1$ we have that $\langle\{v_1,\ldots,v_6\}\rangle$ contains neither $[v_3,v_1]$, $[v_3,v_5]$ nor $[v_3,v_6]$.
Note that $[v_2,v_6]$, $[v_4,v_6]$ may be contained in $\langle\{v_1,\ldots,v_6\}\rangle$. Therefore, $G_1\in \mathcal{F}_6$.

Assume that $L(T)=7$ and $G_1\notin \mathcal{F}_6$.
Denote by $\{v_1,\ldots,v_7\}$ the vertices in $T$ such that $T=\bigcup_{i=1}^{7}[v_i,v_{i+1}]$ with $v_8:=v_1$. Without loss of generality we can assume that $x\in[v_1,v_2]$, $y\in[v_4,v_5]$ and $p=v_3$. Since $d_{G_1\uplus G_2}(x,y)=3$, we have that $\langle\{v_1,\ldots,v_7\}\rangle$ contains neither $[v_1,v_4]$, $[v_1,v_5]$, $[v_2,v_4]$ nor $[v_2,v_5]$; besides, since $d_{G_1\uplus G_2}(p,[yz]\cup[zx])>1$ we have that $\langle\{v_1,\ldots,v_7\}\rangle$ contains neither $[v_3,v_1]$, $[v_3,v_5]$, $[v_3,v_6]$ nor $[v_3,v_7]$.
Since $G_1\notin \mathcal{F}_6$, $[v_1,v_6]$ and $[v_5,v_7]$ are not contained in $\langle\{v_1,\ldots,v_7\}\rangle$.
Note that $[v_2,v_6]$, $[v_2,v_7]$, $[v_4,v_6]$, $[v_4,v_7]$ may be contained in $\langle\{v_1,\ldots,v_7\}\rangle$. Hence, $G_1\in \mathcal{F}_7$.

Assume that $L(T)=8$ and $G_1\notin \mathcal{F}_6\cup\mathcal{F}_7$.
Denote by $\{v_1,\ldots,v_8\}$ the vertices in $T$ such that $T=\bigcup_{i=1}^{8}[v_i,v_{i+1}]$ with $v_9:=v_1$. Without loss of generality we can assume that $x\in[v_1,v_2]$, $y\in[v_4,v_5]$ and $p=v_3$. Since $d_{G_1\uplus G_2}(x,y)=3$, we have that $\langle\{v_1,\ldots,v_8\}\rangle$ contains neither $[v_1,v_4]$, $[v_1,v_5]$, $[v_2,v_4]$ nor $[v_2,v_5]$; besides, since $d_{G_1\uplus G_2}(p,[yz]\cup[zx])>1$ we have that $\langle\{v_1,\ldots,v_8\}\rangle$ contains neither $[v_3,v_1]$, $[v_3,v_5]$, $[v_3,v_6]$, $[v_3,v_7]$ nor $[v_3,v_8]$.
Since $G_1\notin \mathcal{F}_6\cup\mathcal{F}_7$, $[v_1,v_6]$, $[v_1,v_7]$, $[v_5,v_7]$, $[v_5,v_8]$ and $[v_6,v_8]$ are not contained in $\langle\{v_1,\ldots,v_8\}\rangle$.
Since $T$ is a geodesic triangle we have that $z\in\{v_{6,7},v_7,v_{7,8}\}$ with $v_{6,7}$ and $v_{7,8}$ the midpoints of $[v_6,v_7]$ and $[v_7,v_8]$, respectively.
If $z=v_7$ then $\langle\{v_1,\ldots,v_8\}\rangle$ contains neither $[v_2,v_7]$ nor $[v_4,v_7]$.
Note that $[v_2,v_6]$, $[v_2,v_8]$, $[v_4,v_6]$, $[v_4,v_8]$ may be contained in $\langle\{v_1,\ldots,v_8\}\rangle$.
If $z=v_{6,7}$ then $\langle\{v_1,\ldots,v_8\}\rangle$ contains neither $[v_2,v_6]$ nor $[v_2,v_7]$.
Note that $[v_2,v_8]$, $[v_4,v_6]$, $[v_4,v_7]$, $[v_4,v_8]$ may be contained in $\langle\{v_1,\ldots,v_8\}\rangle$.
By symmetry, we obtain an equivalent result for $z=v_{7,8}$. Therefore, $G_1\in \mathcal{F}_8$.

Assume that $L(T)=9$ and $G_1\notin \mathcal{F}_6\cup\mathcal{F}_7\cup\mathcal{F}_8$.
Denote by $\{v_1,\ldots,v_9\}$ the vertices in $T$ such that $T=\bigcup_{i=1}^{9}[v_i,v_{i+1}]$ with $v_{10}:=v_1$. Without loss of generality we can assume that $x\in[v_1,v_2]$, $y\in[v_4,v_5]$ and $p=v_3$. Since $d_{G_1\uplus G_2}(x,y)=3$, we have that $\langle\{v_1,\ldots,v_9\}\rangle$ contains neither $[v_1,v_4]$, $[v_1,v_5]$, $[v_2,v_4]$ nor $[v_2,v_5]$; besides, since $d_{G_1\uplus G_2}(p,[yz]\cup[zx])>1$ we have that $\langle\{v_1,\ldots,v_9\}\rangle$ contains neither $[v_3,v_1]$, $[v_3,v_5]$, $[v_3,v_6]$, $[v_3,v_7]$, $[v_3,v_8]$ nor $[v_3,v_9]$.
Since $T$ is a geodesic triangle we have that $z$ is the midpoint of $[v_7,v_8]$.
Since $d_{G_1\uplus G_2}(y,z)=d_{G_1\uplus G_2}(z,x)=3$, we have that $\langle\{v_1,\ldots,v_9\}\rangle$ contains neither $[v_1,v_7]$, $[v_1,v_8]$, $[v_2,v_7]$, $[v_2,v_8]$, $[v_4,v_7]$, $[v_4,v_8]$, $[v_5,v_7]$ nor $[v_5,v_8]$.
Since $G_1\notin \mathcal{F}_6\cup\mathcal{F}_7\cup\mathcal{F}_8$, $[v_1,v_6]$, $[v_5,v_9]$, $[v_6,v_8]$, $[v_6,v_9]$ and $[v_7,v_9]$ are not contained in $\langle\{v_1,\ldots,v_9\}\rangle$.
Note that $[v_2,v_6]$, $[v_2,v_9]$, $[v_4,v_6]$, $[v_4,v_9]$ may be contained in $\langle\{v_1,\ldots,v_9\}\rangle$.
Hence, $G_1\in \mathcal{F}_9$.

\smallskip

Finally, one can check that if $G_1\in \mathcal{F}$ or $G_2\in \mathcal{F}$, then $\d(G_1\uplus G_2)=3/2$, by following the previous arguments.
\end{proof}

These results allow to compute, in a simple way, the hyperbolicity constant of every graph join:

\begin{theorem}\label{th:hypJoin}
Let $G_1,G_2$ be any two graphs. Then,
\[
\d(G_1\uplus G_2)=\left\{
\begin{array}{ll}
0,&\text{if } G_i \simeq E_1 \text{ and } G_j \simeq E_n \text{ for } i\neq j \text{ and } n\in \NN,\\
3/4,&\text{if } G_i\simeq E_1 \text{ and } \D_{G_j}=1 \text{ for } i\neq j,\\
1,&\text{if } G_i\simeq E_1 \text{ and } 1< \diam^* G_j\le2 \text{ for } i\neq j;\text{ or}\\
& n_i\ge2 \text{ and }  \diam G_i\le2 \text{ or}\\
& G_i \text{ is an empty graph for } i=1,2;\\
3/2,&\text{if } G_1\in \mathcal{F} \text{ or } G_2\in\mathcal{F},\\
5/4,&\text{otherwise}.
\end{array}
\right.
\]
\end{theorem}

\begin{corollary}\label{c:emptyUplus}
Let $G$ be any graph. Then,
\[
\d(E_1\uplus G)=\left\{
\begin{array}{ll}
0,\quad &\mbox{if } \diam^* G=0,\\
3/4,\quad &\mbox{if } \diam^* G=1,\\
1,\quad &\mbox{if } 1< \diam^* G\le2,\\
5/4,\quad &\mbox{if } \diam^* G>2 \text{ and } G\notin \mathcal{F},\\
3/2,\quad &\mbox{if } G\in \mathcal{F}.
\end{array}
\right.
\]
\end{corollary}

\section{Hyperbolicity of corona of two graphs}\label{Sect4}
In this section we study the hyperbolicity of the corona of two graphs, defined by Frucht and Harary in 1970, see \cite{FH}.

\begin{definition}\label{def:crown}
Let $G_1$ and $G_2$ be two graphs with $V(G_1)\cap V(G_2)=\emptyset$. The \emph{corona} of $G_1$ and $G_2$, denoted by $G_1\diamond G_2$, is defined as the graph obtained by taking one copy of $G_1$ and a copy of $G_2$ for each vertex $v\in V(G_1)$, and then joining each
vertex $v\in V(G_1)$ to every vertex in the $v$-th copy of $G_2$.
\end{definition}

From the definition, it clearly follows that the corona product of two graphs is a non-commutative and non-associative operation.
Figure \ref{fig:crown} show the corona of two graphs.

\begin{figure}[h]
\centering
\scalebox{1.2}
{\begin{pspicture}(-.75,-1.9)(6.5,1.9)
\psline[linewidth=0.01cm,dotsize=0.07055555cm 2.5]{-*}(-.5,-.5)(.5,-.5)
\psline[linewidth=0.01cm,dotsize=0.07055555cm 2.5]{-*}(.5,-.5)(.5,.5)
\psline[linewidth=0.01cm,dotsize=0.07055555cm 2.5]{-*}(.5,.5)(-.5,.5)
\psline[linewidth=0.01cm,dotsize=0.07055555cm 2.5]{-*}(-.5,.5)(-.5,-.5)
\uput[0](1,0){$\diamond$}
\psline[linewidth=0.01cm,dotsize=0.07055555cm 2.5]{-*}(1.5,-.5)(2.65,-.5)
\psline[linewidth=0.01cm,dotsize=0.07055555cm 2.5]{-*}(2.65,-.5)(2.075,.5)
\psline[linewidth=0.01cm,dotsize=0.07055555cm 2.5]{-*}(2.075,.5)(1.5,-.5)
\uput[0](2.8,0){\large{=}}
\psline[linewidth=0.01cm,dotsize=0.07055555cm 2.5]{-}(3.25,-.75)(4.2,-.7)(3.75,-1.61)(3.25,-.75)
\psline[linewidth=0.01cm,dotsize=0.07055555cm 2.5]{-}(5.8,-.7)(6.75,-.75)(6.25,-1.61)(5.8,-.7)
\psline[linewidth=0.01cm,dotsize=0.07055555cm 2.5]{-}(5.8,.7)(6.75,.75)(6.25,1.61)(5.8,.7)
\psline[linewidth=0.01cm,dotsize=0.07055555cm 2.5]{-}(3.25,.75)(4.2,.7)(3.75,1.61)(3.25,.75)
\psline[linewidth=0.01cm,dotsize=0.07055555cm 2.5]{-}(4.5,-.5)(5.5,-.5)(5.5,.5)(4.5,.5)(4.5,-.5)
\psline[linewidth=0.01cm,dotsize=0.07055555cm 2.5]{*-*}(4.5,-.5)(3.25,-.75)
\psline[linewidth=0.01cm,dotsize=0.07055555cm 2.5]{-*}(4.5,-.5)(4.2,-.7)
\psline[linewidth=0.01cm,dotsize=0.07055555cm 2.5]{-*}(4.5,-.5)(3.75,-1.61)
\psline[linewidth=0.01cm,dotsize=0.07055555cm 2.5]{*-*}(5.5,-.5)(5.8,-.7)
\psline[linewidth=0.01cm,dotsize=0.07055555cm 2.5]{-*}(5.5,-.5)(6.75,-.75)
\psline[linewidth=0.01cm,dotsize=0.07055555cm 2.5]{-*}(5.5,-.5)(6.25,-1.61)
\psline[linewidth=0.01cm,dotsize=0.07055555cm 2.5]{*-*}(5.5,.5)(5.8,.7)
\psline[linewidth=0.01cm,dotsize=0.07055555cm 2.5]{-*}(5.5,.5)(6.75,.75)
\psline[linewidth=0.01cm,dotsize=0.07055555cm 2.5]{-*}(5.5,.5)(6.25,1.61)
\psline[linewidth=0.01cm,dotsize=0.07055555cm 2.5]{*-*}(4.5,.5)(3.25,.75)
\psline[linewidth=0.01cm,dotsize=0.07055555cm 2.5]{-*}(4.5,.5)(4.2,.7)
\psline[linewidth=0.01cm,dotsize=0.07055555cm 2.5]{-*}(4.5,.5)(3.75,1.61)
\end{pspicture}}
\caption{Corona of two graphs $C_4 \diamond C_3$.} \label{fig:crown}
\end{figure}
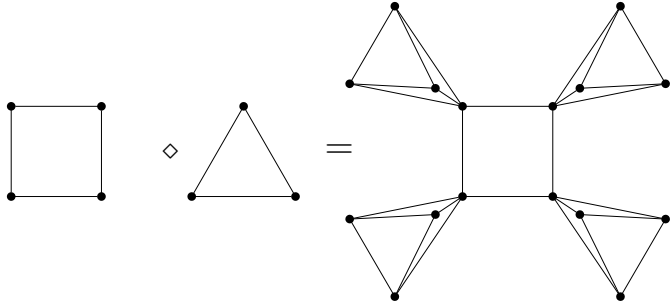

Many authors deal just with corona of finite graphs; however, our results hold for finite or infinite graphs.

If $G$ is a connected graph, we say that $v\in V(G)$ is a \emph{connection vertex} if
$G \setminus \{v\}$ is not connected.

Given a connected graph $G$, a family of subgraphs $\{G_n\}_{n\in \Lambda}$ of
$G$ is a \emph{T-decomposi\-tion} of $G$ if $\cup_n G_n=G$ and $G_n\cap G_m$ is either a connection vertex or the empty set for each $n\neq m$.

We will need the following result (see \cite[Theorem 5]{BRSV2}),
which allows to obtain global information about the
hyperbolicity of a graph from local information.

\begin{theorem}
\label{t:treedec}
Let $G$ be any connected graph and let $\{G_n\}_n$ be any
T-decomposition of $G$. Then $\d(G)=\sup_n \d(G_n)$.
\end{theorem}

We remark that the corona $G_1\diamond G_2$ of two graphs is connected if and only if $G_1$ is connected.

The following result characterizes the hyperbolicity of the corona of two graphs and provides the precise value of its hyperbolicity constant.

\begin{theorem}\label{th:corona}
Let $G_1,G_2$ be any two graphs. Then $\d(G_1\diamond G_2)=\max\{\d(G_1),\d(E_1\uplus G_2)\}$.
\end{theorem}

\begin{proof}
Assume first that $G_1$ is connected.
The formula follows from Theorem \ref{t:treedec}, since $\{G_1\}\cup \big\{\{v\}\uplus G_2\big\}_{v\in V(G_1)}$ is a T-decomposition of $G_1\diamond G_2$.
Finally, note that if $G_1$ is a non-connected graph, then we can apply the previous argument to each connected component.
\end{proof}

Note that Corollary \ref{c:emptyUplus} provides the precise value of $\d(E_1\uplus G_2)$.

\begin{corollary}\label{cor:sup}
Let $G_1,G_2$ be any two graphs. Then $G_1\diamond G_2$ is hyperbolic if and only if $G_1$ is hyperbolic.
\end{corollary}

\begin{proof}
By Theorem \ref{th:corona} we have $\d(G_1\diamond G_2)=\max\{\d(G_1),\d(E_1\uplus G_2)\}$. Then, by Corollary \ref{cor:SP} we have $\d(G_1) \le \d(G_1\diamond G_2) \le \max\{\d(G_1),3/2\}$.
\end{proof}

\end{document}